\pdfoutput=1
\documentclass[reqno]{amsart}
\usepackage{amssymb}
\usepackage{amsmath}
\usepackage{amsthm}
\usepackage{mathrsfs}
\usepackage{amsfonts}
\usepackage{mathtools}
\usepackage{graphicx}
\usepackage{comment}

\usepackage{hyperref}
\hypersetup{colorlinks}
\usepackage{tikz}
\usetikzlibrary{arrows}

\usepackage[english]{babel}

\usepackage[margin=1.4in, centering]{geometry}

\DeclareSymbolFont{bbold}{U}{bbold}{m}{n}
\DeclareSymbolFontAlphabet{\mathbbold}{bbold}
\newtheorem{thm}{Theorem}[section]
\newtheorem{prop}[thm]{Proposition}

\newtheorem{cor}[thm]{Corollary}

\newtheorem{obs}[thm]{Observation}
\theoremstyle{definition}
\newtheorem{defn}[thm]{Definition}
\newtheorem{exmp}[thm]{Example}
\theoremstyle{remark}
\newtheorem{rmk}[thm]{Remark}

\newcommand{\R}{\mathbb{R}}
\newcommand{\C}{\mathbb{C}}

\newcommand{\N}{\mathbb{N}}
\newcommand{\Z}{\mathbb{Z}}

\newcommand{\F}{\mathbb{F}}
\newcommand{\T}{\mathbb{T}}

\newcommand{\FF}{\mathscr{P}}
\newcommand{\GG}{\mathscr{C}}
\newcommand{\CC}{\mathcal{C}}
\newcommand{\PP}{\mathcal{P}}

\newcommand{\OO}{\mathcal{O}}
\newcommand{\EE}{\mathcal{E}}
\newcommand{\DD}{\mathcal{D}}
\newcommand{\LL}{\mathcal{L}}

\newcommand{\B}{\mathcal{B}}

\DeclarePairedDelimiter\abs{\lvert}{\rvert}
\DeclarePairedDelimiter\norm{\lVert}{\rVert}

\makeatletter
\let\oldabs\abs
\def\abs{\@ifstar{\oldabs}{\oldabs*}}

\let\oldnorm\norm
\def\norm{\@ifstar{\oldnorm}{\oldnorm*}}
\makeatother

\newcommand{\overbar}[1]{{\mkern 1.5mu\overline{\mkern-1.5mu#1\mkern-1.5mu}\mkern 1.5mu}}

\newcommand{\defeq}{\vcentcolon=}

\title[A Short Proof of an  Identity Related to Type IV Superorthogonality]{A Short Proof of an Identity Related to Type IV Superorthogonality}

\author{Yixuan Pang}

\address[Y. Pang]{Department of Mathematics, University of Pennsylvania, Philadelphia, PA 19104}
\email{pyixuan@sas.upenn.edu}

\begin{document}

\begin{abstract}
    We provide a much shorter but even more powerful proof of an algebraic identity, which can be used to establish the direct and the converse inequality under Type IV superorthogonality. As an application, we obtain the optimal order of the formal constant in the direct inequality. This order turns out to be also sharp for Type III superorthogonality. When $p\geq 2$, our result recovers the optimal order of the constant in the Burkholder-Gundy inequality in martingale theory, and also recovers the currently best order of the constant in the reverse Littlewood-Paley inequality. We also introduce variants of Type IV superorthogonality, under which we prove the converse inequality.
\end{abstract}

\maketitle

\tableofcontents 

\section{Introduction}
Square function estimates appear almost everywhere in harmonic analysis. Let $(X,\mu)$ be a $\sigma$-finite measure space, and let $\{f_l\}_{1\leq l\leq L}$ be a finite family of complex-valued functions in $L^{2r}(X,d\mu)$, where $r\in \N^+$. We consider the direct inequality
\begin{align}\label{direct}
    \norm{\sum_{l=1}^L f_l}_{2r} \lesssim_r \norm{\left(\sum_{l=1}^L\abs{f_l}^2\right)^{\frac{1}{2}}}_{2r}
\end{align}
and the converse inequality
\begin{align}\label{converse}
    \norm{\left(\sum_{l=1}^L\abs{f_l}^2\right)^{\frac{1}{2}}}_{2r}
    \lesssim_r 
    \norm{\sum_{l=1}^L f_l}_{2r}.
\end{align}
As $r\in \N^+$, we can expand both sides algebraically. This leads to the concept of ``superorthogonality'', which tells us to what extent elements in the family $\{f_l\}_{1\leq l\leq L}$ satisfy the vanishing property:
\begin{align}\label{condition}
    \int_X f_{l_1}\overbar{f_{l_2}} \cdots f_{l_{2r-1}}\overbar{f_{l_{2r}}} d\mu = 0.
\end{align}
Such type of condition can unify and reinterpret many classical results in harmonic analysis, and is very useful when dealing with some discrete operators. Notably, \cite{PierceOnSuperorthogonality}\cite{gressman2023new} introduced a hierarchy of superorthogonality:
\begin{itemize}
    \item Type I$^*$: (\ref{condition}) holds whenever $(l_{2k-1})_{1\leq k\leq r}$ is not a permutation of $(l_{2k})_{1\leq k\leq r}$.
    \item Type I: (\ref{condition}) holds whenever some $l_j$ appears an odd time in $l_1,\dots, l_{2r}$.
    \item Type II: (\ref{condition}) holds whenever some $l_j$ appears precisely once in $l_1,\dots, l_{2r}$.
    \item Type III: (\ref{condition}) holds whenever some $l_j$ is strictly larger all other indices in $l_1,\dots, l_{2r}$.
    \item Type IV: (\ref{condition}) holds whenever $l_1,\dots, l_{2r}$ are all distinct.
\end{itemize}
The assumptions above on the family $\{f_l\}_{1\leq l\leq L}$ get weaker from top to bottom, and the reader may consult \cite{PierceOnSuperorthogonality}\cite{gressman2023new} for detailed historical remarks. Recently, \cite{gressman2023new} showed that Type IV superorthogonality implies the direct inequality (\ref{direct}). For the converse inequality (\ref{converse}), as $\ell^{2}\xhookrightarrow{} \ell^{2r}$, a necessary condition is:
\begin{align}\label{necessary}
    \norm{\left(\sum_{l=1}^L\abs{f_l}^{2r}\right)^{\frac{1}{2r}}}_{2r}
    \lesssim_r 
    \norm{\sum_{l=1}^L f_l}_{2r}.
\end{align}

In practice, (\ref{necessary}) is much weaker and easier to prove than (\ref{converse}). Therefore, we can raise the question: Suppose (\ref{necessary}) holds, then what is the weakest type of superorthogonality we need to establish (\ref{converse})?

A classical work due to Paley, as explained in \cite{PierceOnSuperorthogonality}, proved that Type III suffices. Recently, \cite{2023arXiv231201717Z} showed that we actually only need Type IV, and generalized everything to the vector-valued version, using a ``folding-the-circle'' trick. For simplicity, here we only record the main result in \cite{2023arXiv231201717Z} in the complex-valued case:
\begin{thm}[\cite{2023arXiv231201717Z}]\label{main thm}
    Fix $r\in\N^+$ and a $\sigma$-finite measure space $(X,d\mu)$. Assume that $\{f_l\}_{1\leq l\leq L}$ is a finite family of complex-valued functions in $L^{2r}(X,d\mu)$, which satisfies Type IV superorthogonality:
    \begin{align}
        \int_X f_{l_1}\overbar{f_{l_2}} \cdots f_{l_{2r-1}}\overbar{f_{l_{2r}}} d\mu = 0 \,\,\, \text{whenever $l_1,\dots, l_{2r}$ are all distinct}.
    \end{align}
    Then we have
    \begin{align*}
        \norm{\left(\sum_{l=1}^L\abs{f_l}^2\right)^{\frac{1}{2}}}_{2r}
        \lesssim_r 
        \norm{\sum_{l=1}^L f_l}_{2r},
    \end{align*}
    provided that (\ref{necessary}) holds.
\end{thm}
A key ingredient in the proof of Theorem \ref{main thm} in \cite{2023arXiv231201717Z} is an algebraic identity, whose original proof is inspired by arguments in \cite{gressman2023new} and doesn't provide precise information about the coefficients.

In this paper, we provide a much shorter but even more powerful proof of the algebraic identity, see Theorem \ref{main result} and the remarks following it. Our proof is inspired by enumerative combinatorics and provides precise information of all the coefficients, by which we obtain the sharp order of the formal constant in the direct inequality. This order turns out to be also sharp for Type III superorthogonality. Our result recovers the optimal order of the constant in the Burkholder-Gundy inequality in martingale theory, and also recovers the currently best order of the constant in the reverse Littlewood-Paley inequality when $p\geq 2$.

\subsection*{Outline of the paper} In Section \ref{sec:notation}, we list our basic notation; in Section \ref{sec:concepts}, we introduce the key combinatorial concepts, which lays the foundation of the paper; in Section \ref{sec:main result}, we state and prove the main theorem; in Section \ref{sec:coefficients}, we explicitly calculate the coefficients using combinatorial tools; in Section \ref{sec:application}, we offer various applications, especially in martingale theory and Littlewood-Paley theory; in Section \ref{sec:directions}, we discuss some possible further directions.

\subsection*{Acknowledgements}
The author would like to thank Jianghao Zhang for introducing him to superorthogonality, for helpful math discussions, and for many comments on an early manuscript. The author is also grateful to Prof. Philip Gressman for many valuable suggestions and conversations.

\section{Basic notation}\label{sec:notation}
\begin{itemize}
    \item $[n]\defeq \{1,\dots,n\},\forall\, n\in\N^+$.
    \item $\binom{n}{m} \defeq \frac{n!}{m!(n-m)!},\forall\, 0\leq m \leq n$ ($m,n\in\Z$). This is called the ``binomial coefficients''.
    \item $[x]_{(n)}\defeq x(x-1)(x-2)\cdots(x-n+1),\forall\,n\in\N^+$, $x$ an indeterminate. This is called the ``falling factorials''.
    \item For a finite set $A$, let $\abs{A}$ denote the cardinality of $A$.
    \item For a family of sets $\PP$ with elements denoted by $P$, the notation $\displaystyle \sum_{(l_P)_{P\in\PP}}$ represents the sum over indices $1\leq l_P\leq L$, $P\in\PP$, which are independent with each other. We call it an ``independent'' sum.
    \item For a family of sets $\PP$ with elements denoted by $P$, the notation $\displaystyle \sum_{(l_P)_{P\in\PP}}^*$ represents the sum over indices $1\leq l_P\leq L$, $P\in\PP$, which are all distinct. We call it a ``distinct'' sum.
\end{itemize}

\section{Key concepts}\label{sec:concepts}
All the concepts introduced in this section are natural in combinatorics.
\begin{defn}[Partition]
    A partition $\PP$ of $[n]$ is a collection of non-empty disjoint subsets of $[n]$ with $[n]=\cup_{P\in\PP}P$.
\end{defn}
For simplicity, we adopt the notation:
\begin{itemize}
    \item $\FF$: the family of all partitions $\PP$ of $[n]$.
    \item  $\PP_0\defeq \{\{1\},\cdots, \{n\}\}$.
    \item For $\PP\in\FF$ and each $j\in [n]$, let $\PP(j)$ denote the unique $P\in\PP$ such that $j\in P$.
\end{itemize}

\begin{defn}[Type of partitions]\label{type_partition}
    For any $\PP\in\FF$ and $i\in\N^+$, let $\#_i^\PP\defeq \abs{\{P\in\PP: \abs{P}=i\}}$.
    Define an equivalence relation ``$\sim$'' over $\FF$: For $\PP_1,\PP_2\in \FF$, we say $\PP_1\sim\PP_2$ whenever $\#_i^{\PP_1} = \#_i^{\PP_2}$, $\forall\, i\in\N^+$. Each equivalent class is a \textit{type}. For each $\PP\in\FF$, we can use an unordered tuple $((\PP))\defeq (\abs{P})_{P\in\PP}$ to represent the type of $\PP$.
\end{defn}

\begin{defn}[Partial order]\label{partial order}
    For $\PP_1,\PP_2\in\FF$, we say that $\PP_1$ is \textit{finer} than $\PP_2$ (or $\PP_2$ is \textit{coarser} than $\PP_1$), if for any $P\in\PP_1$, there exists some $P'\in\PP_2$ such that $P\subseteq P'$. In this case we may write ``$\PP_1 \leq \PP_2$'' (or ``$\PP_2 \geq \PP_1$''). We also write ``$\PP_1 < \PP_2$''(or ``$\PP_2 > \PP_1$'') when $\PP_1 \leq \PP_2$ but $\PP_1 \neq \PP_2$.
\end{defn}

\begin{defn}[Chain]\label{chain}
    A nonempty subfamily $\CC \subseteq \FF$ is called a \textit{chain}, if it's a totally ordered set in $(\FF,<)$.
\end{defn}

For simplicity, we adopt the notation:
\begin{itemize}
    \item $\GG$: the collection of all chains $\CC$ in $\FF$.
    \item $\GG(\PP_1,\PP_2)\defeq \{\CC\in\GG: \PP_1\textit{ is the finest in } \CC, \PP_2\textit{ is the coarsest in } \CC\}$.
\end{itemize}

\begin{defn}[Odd/even chain]
    We call $\CC\in\GG$ an \textit{odd} chain if $\abs{\CC}$ is odd, and an \textit{even} chain if $\abs{\CC}$ is even.
\end{defn}

\begin{defn}
    For any $\PP_1\leq\PP_2$ in $(\FF,\leq)$, define:
    \begin{itemize}
        \item $\OO(\PP_1,\PP_2)$: the total number of odd chains in $\GG(\PP_1,\PP_2)$.
        \item $\EE(\PP_1,\PP_2)$: the total number of even chains in $\GG(\PP_1,\PP_2)$.
        \item $\DD(\PP_1,\PP_2)\defeq \OO(\PP_1,\PP_2)-\EE(\PP_1,\PP_2)$.
    \end{itemize}
    In particular, when $\PP_1=\PP_0$, we abbreviate $\DD(\PP_0,\PP)$ to $\DD(\PP)$.
\end{defn}

Before ending this section, we offer a self-evident observation:
\begin{obs}[Symmetry]\label{symmetry}
    If $\PP_1\sim\PP_2$ in $\FF$, then $\DD(\PP_1) = \DD(\PP_2)$.
\end{obs}

\section{The algebraic identity}\label{sec:main result}
In this section we state and prove our main theorem:
\begin{thm}\label{main result}
    Given $n\in \N^+$, suppose  $V_j$, $j\in [n]$, are $n$ arbitrary vector spaces over a field $\F$, and each $V_j$ contains $L$ vectors $v_{j,l_j}$, $1\leq l_j\leq L$. Then for any $\PP_1\in\FF$, we have
    \begin{align}\label{algebraic}
        \sum_{(l_P)_{P\in \PP_1}}^* v_{1,l_{\PP_1(1)}}\otimes\cdots\otimes v_{n,l_{\PP_1(n)}}
        = \sum_{\PP\geq\PP_1}\DD(\PP_1,\PP) \sum_{(l_P)_{P\in \PP}} v_{1,l_{\PP(1)}}\otimes\cdots\otimes v_{n,l_{\PP(n)}}.
    \end{align}
\end{thm}

\begin{proof}[Proof of Theorem \ref{main result}]
    We will induct on $\abs{\PP_1}$. The base case when $\abs{\PP_1} = 1$ is trivially verified.
    
    The basic idea is, whenever we encounter a distinct sum, we replace it with its corresponding independent sum and subtract all ``coarser'' distinct sums. This yields the key observation:
    \begin{align}\label{first_step}
        \sum_{(l_P)_{P\in \PP_1}}^* v_{1,l_{\PP_1(1)}}\otimes\cdots\otimes v_{n,l_{\PP_1(n)}}
        =& \sum_{(l_P)_{P\in \PP_1}} v_{1,l_{\PP_1(1)}}\otimes\cdots\otimes v_{n,l_{\PP_1(n)}}\nonumber\\
        &- \sum_{\PP_2 > \PP_1} \sum_{(l_P)_{P\in \PP_2}}^* v_{1,l_{\PP_2(1)}}\otimes\cdots\otimes v_{n,l_{\PP_2(n)}}.
    \end{align}

    As $\abs{\PP_2}<\abs{\PP_1}$, we can apply induction hypothesis to the distinct sum in (\ref{first_step}) to obtain
    \begin{align*}
        &\sum_{(l_P)_{P\in \PP_1}}^* v_{1,l_{\PP_1(1)}}\otimes\cdots\otimes v_{n,l_{\PP_1(n)}}\\
        =& \sum_{(l_P)_{P\in \PP_1}} v_{1,l_{\PP_1(1)}}\otimes\cdots\otimes v_{n,l_{\PP_1(n)}}\nonumber
        - \sum_{\PP_2 > \PP_1} \left(\sum_{\PP\geq\PP_2}\DD(\PP_2,\PP) \sum_{(l_P)_{P\in \PP}} v_{1,l_{\PP(1)}}\otimes\cdots\otimes v_{n,l_{\PP(n)}}\right)\\
        =& \sum_{(l_P)_{P\in \PP_1}} v_{1,l_{\PP_1(1)}}\otimes\cdots\otimes v_{n,l_{\PP_1(n)}}\nonumber
        - \left(\sum_{\PP_2: \PP\geq \PP_2 > \PP_1}\DD(\PP_2,\PP)\right) \sum_{(l_P)_{P\in \PP}} v_{1,l_{\PP(1)}}\otimes\cdots\otimes v_{n,l_{\PP(n)}}\\
        \overset{\triangle}{=}&\sum_{(l_P)_{P\in \PP_1}} v_{1,l_{\PP_1(1)}}\otimes\cdots\otimes v_{n,l_{\PP_1(n)}}\nonumber + 
        \sum_{\PP > \PP_1}\DD(\PP_1,\PP) \sum_{(l_P)_{P\in \PP}} v_{1,l_{\PP(1)}}\otimes\cdots\otimes v_{n,l_{\PP(n)}}\\
        =&\sum_{\PP \geq \PP_1}\DD(\PP_1,\PP) \sum_{(l_P)_{P\in \PP}} v_{1,l_{\PP(1)}}\otimes\cdots\otimes v_{n,l_{\PP(n)}}.
    \end{align*}
    Here in $\triangle$ we naturally correspond $\GG(\PP_2,\PP)$ to $\GG(\PP_1,\PP)$ by adding $\PP_2>\PP_1$, only with the parity of chains changed, which is why ``$-$'' becomes ``$+$''. So the proof is completed.
\end{proof}

There are several noteworthy features of Theorem \ref{main result} and its proof:
\begin{enumerate}
    \item Our identity (\ref{algebraic}) is actually more powerful than that in \cite[Theorem 4]{2023arXiv231201717Z}, as we give precise characterization of \textit{all} the coefficients, which helps us better understand the formal constant in square function estimates (Section \ref{subsec:constant} and \ref{subsec:applications}). Also, we don't require $\PP_1 = \PP_0$, which offers additional flexibility in other applications (Section \ref{subsec:variants}).
    \item Despite being more powerful, our proof of (\ref{algebraic}) is even much simpler. In fact, the original proof in \cite{2023arXiv231201717Z} takes around three pages, while our proof is within just half a page. Notably, the original proof involves three successive induction procedures, while our proof only uses one induction procedure.
    \item Besides being more powerful and concise, our proof builds a bridge between superorthogonality and enumerative combinatorics in a natural and elegant way, which should be helpful in understanding the intrinsic structure of superorthogonality.
\end{enumerate}

Before ending this section, we point out that Theorem \ref{main result} immediately yields an identity for multilinear functions, which is easier to use in practice.
\begin{cor}\label{identity_multilinear}
    Given $n\in\N^+$, suppose $V_j$, $j\in[n]$, are $n$ arbitrary vector spaces over a field $\F$, and each $V_j$ contains $L$ vectors $v_{j,l_j}$, $1\leq l_j\leq L$. Then for any $\PP_1\in\FF$ and any $n$-linear function $\Lambda: V_1\times\cdots\times V_n\rightarrow \F$, we have
    \begin{align*}
        \sum_{(l_P)_{P\in\PP_1}}^* \Lambda(v_{1,\PP_1(1)}, \cdots, v_{n,\PP_1(n)}) = \sum_{\PP\geq\PP_1}
        \DD(\PP_1,\PP) \sum_{(l_P)_{P\in\PP}} \Lambda(v_{1,\PP(1)}, \cdots, v_{n,\PP(n)}).
    \end{align*}
\end{cor}

\section{Calculation of the coefficients}\label{sec:coefficients}
For Type IV superorthogonality, \cite{gressman2023new} first established the direct inequality with a formal constant $2^{1/2}((2r)!-1)^{1/2}$, and raised a question: can we improve this formal constant significantly in general?

As pointed out in the previous section, \cite[Theorem 4]{2023arXiv231201717Z} doesn't shed much light on what general coefficients are, which makes it hard to bound the formal constants explicitly. Our alternative proof beats this drawback successfully: the coefficients are exactly $\DD(\PP_1,\PP)$. If we can calculate $\DD(\PP_1,\PP)$ explicitly, then we may get a better formal constant for the direct inequality, thus answering the question in \cite{gressman2023new} affirmatively. This is indeed the case, see Section \ref{subsec:constant}.

The most important case is $\DD(\PP)$, and we calculate it using combinatorial tools.

\begin{defn}[Stirling number of the second kind]
    A ``Stirling number of the second kind''(or ``Stirling partition number'') is the number of ways to partition a set of $n$ objects into $k$ non-empty subsets, and is denoted by $S(n,k)$.
\end{defn}
\begin{rmk}
    The $n$ objects are labelled, while the $k$ subsets are unlabelled.
\end{rmk}
For a systematic discussion of $S(n,k)$, one may consult \cite{mezo2019combinatorics} or Chapter V of \cite{comtet1974advanced}. Here we only record the property we will use without proof:
\begin{prop}[An identity involving the falling factorials]\label{property}
    \begin{align}\label{factorial_formula}
        \sum_{k=1}^n S(n,k)[x]_{(k)} = x^n.
    \end{align}
\end{prop}
\begin{cor}
    We have
    \begin{align}\label{variant_formula}
        \sum_{k=1}^n S(n,k)(-1)^{k-1}(k-1)! = 0.
    \end{align}
\end{cor}
\begin{proof}
    Divide both sides of (\ref{factorial_formula}) by $x$, and then set the indeterminate $x$ to be $0$.
\end{proof}

Our main result in this section is:
\begin{thm}\label{general_coefficient}
    For $n\in\N^+$, let $\PP\in\FF$ with $\abs{\PP}=m$, $m\in[n]$. Suppose $((\PP)) = (n_i)_{i\in[m]}$, then $ \DD(\PP) = (-1)^{n-m}\prod_{i=1}^m(n_i-1)!$.
\end{thm}

\begin{proof}[Proof  of Theorem \ref{general_coefficient}]
    When $n=1, 2$, things are trivial. For general $n$, we will use induction. Since we will encounter multiple tuples $(n_i)_{i\in[m]}$, it would be convenient to emphasize the dependence on $(n_i)_{i\in[m]}$: For each $((\PP))=(n_i)_{i\in[m]}$, let $\DD_{(n_i)_{i\in[m]}} \defeq \DD(\PP)$.

    Suppose our result holds for all $k<n$ for some $n\in\N^+$. By classification according to the first step, we have the following recursive identity:
    \begin{align}\label{recursion2}
        \DD_{(n_i)_{i\in[m]}} = -\sum_{\substack{1\leq k_i\leq n_i,\forall\, i\in[m]\\(k_i)_{i\in[m]}\neq(n_i)_{i\in[m]}}} S(n_1,k_1)\cdots S(n_m,k_m) \DD_{(k_i)_{i\in[m]}}
    \end{align}
    The combinatorial interpretation is that, for each $1\leq k_i\leq n_i$ ($i\in[m]$), we have $S(n_i,k_i)$ ways of grouping $n_i$ objects into $k_i$ objects, and different $i$'s are independent with each other, so the total number of ways is $S(n_1,k_1)\cdots S(n_m,k_m)$ by multiplication principle. After this the roles of odd chains and even chains are reversed, and the remaining combinatorial features only depend on $(k_i)_{i\in[m]}$.

    Note that $\sum_{i\in[m]}n_i = n$, $\sum_{i\in[m]}k_i < n$, so by our induction hypothesis on $\DD_{(k_i)_{i\in[m]}}$ ($1\leq k_i\leq n_i,\forall\, i\in[m]$ with $(k_i)_{i\in[m]}\neq(n_i)_{i\in[m]}$), (\ref{recursion2}) becomes
    \begin{align}\label{plug_in2}
        \DD_{(n_i)_{i\in[m]}} = &-\sum_{\substack{1\leq k_i\leq n_i,\forall\, i\in[m]\\(k_i)_{i\in[m]}\neq(n_i)_{i\in[m]}}} S(n_1,k_1)\cdots S(n_m,k_m) (-1)^{k_1+\cdots+k_m - m}\prod_{i=1}^m(k_i-1)!\nonumber\\
        = &-\sum_{k_1=1}^{n_1}\cdots\sum_{k_m=1}^{n_m} S(n_1,k_1)\cdots S(n_m,k_m) (-1)^{k_1+\cdots+k_m - m}\prod_{i=1}^m(k_i-1)!\nonumber\\ 
        &\quad\quad + 
        S(n_1,n_1)\cdots S(n_m,n_m) (-1)^{n_1+\cdots+n_m - m}\prod_{i=1}^m(n_i-1)!\nonumber\\
        =& -(-1)^m\prod_{i=1}^m\left[\sum_{k_i=1}^{n_i} S(n_i,k_i)(-1)^{k_i}(k_i-1)!\right]\nonumber\\
        &\quad\quad + S(n_1,n_1)\cdots S(n_m,n_m) (-1)^{n - m}\prod_{i=1}^m(n_i-1)!.
    \end{align}

    Plugging (\ref{variant_formula}) and the fact that $S(n_i,n_i)=1$ ($i\in[m]$) in (\ref{plug_in2}) yields
    \begin{align*}
        \DD_{(n_i)_{i\in[m]}} = (-1)^{n - m}\prod_{i=1}^m(n_i-1)!
    \end{align*}
    which finishes our inductive arguments and completes our proof.
\end{proof}

By Theorem \ref{general_coefficient}, we can easily recover \cite[Theorem 4]{2023arXiv231201717Z} by noting that
\begin{enumerate}
    \item $\DD(\PP_0) = (-1)^{n-n}\prod_{i=1}^n(1-1)! = 1$;
    \item  $\DD(\PP) = (-1)^{n-\frac{n}{2}}\prod_{i=1}^{\frac{n}{2}}(2-1)! = (-1)^{\frac{n}{2}}$ when $\#_2^\PP = \frac{n}{2}$.
\end{enumerate}

Finally, for applications (Section \ref{subsec:variants}), we are especially interested in the following case:
\begin{defn}[Good partition]\label{good_part}
    We say $\PP\in\FF$ is a ``good partition'', if $\abs{P}=1$ or $\abs{P}=2$ for all $P\in\PP$.
\end{defn}

\begin{prop}\label{useful}
    In Theorem \ref{main result}, suppose $\PP_1\leq \PP$ in $(\FF,<)$ are both good partitions. If $\#_2^\PP - \#_2^{\PP_1}= m$, then $\DD(\PP_1,\PP) = (-1)^m$.
\end{prop}
\begin{proof}
    We argue by directly exploiting the combinatorial meaning of $\DD(\PP_1,\PP)$.
    
    Since $\PP_1\leq \PP$ in $(\FF,<)$ are both good partitions, we can only unite pairs of singletons in $\PP_1$ to get $\PP$, leaving those $\{P\in\PP_1: \abs{P}=2\}$ fixed in all $\CC\in\GG(\PP_1,\PP)$. Thus without loss of generality, we can assume that $\abs{P}=1$ for all $P\in\PP_1$, i.e., $\PP_1 = \PP_0$. 
    
    Moreover, we can also ignore those singletons $P\in\PP_0$ fixed in all $\CC\in\GG(\PP_0,\PP)$, so we may further assume that $n$ is even and $\abs{P}=2$, $\forall\,P\in\PP$. 
    
    Now $m=\#_2^\PP=\frac{n}{2}$ and so $\DD(\PP) = (-1)^{\frac{n}{2}} = (-1)^m$.
\end{proof}

\section{Applications}\label{sec:application}
\subsection{Formal constants in direct inequalities}\label{subsec:constant}\phantom{x}

Corollary \ref{identity_multilinear} and Theorem \ref{general_coefficient} enable us to obtain sharp bounds of the formal constants in direct inequalities. 

Along the way we will need a classical tool from combinatorics and number theory:
\begin{defn}[The partition function]
    For $n\in\N^+$, let $p(n)$ denote the total number of ways to write $n$ as a sum of positive integers.
\end{defn}

Recall our definition of ``types'' (Definition \ref{type_partition}), we know the number of all possible types in $\FF$ is exactly $p(n)$.

The partition function $p(n)$ enjoys lots of interesting properties, for which one can refer to Chapter II of \cite{comtet1974advanced}. However, for our purpose, we only record an upper bound for $p(n)$ here:
\begin{prop}[\cite{de2009simple} A simple upper bound for $p(n)$]\label{part_upper}
    For all $n\in\N^+$, we have
    \begin{align}
        p(n) < \frac{e^{c\sqrt{n}}}{n^{\frac{3}{4}}}
    \end{align}
    where $c = 2\sqrt{\zeta(2)} = \pi\sqrt{\frac{2}{3}} = 2.56509966\dots$ with $\zeta$ be the Riemann zeta function.
\end{prop}

\begin{thm}\label{effective direct}
    Fix $r\in\N^+$ and a $\sigma$-finite measure space $(X,d\mu)$. Assume that $\{f_l\}_{1\leq l\leq L}$ is a finite family of complex-valued functions in $L^{2r}(X,d\mu)$, which satisfies Type IV superorthogonality:
    \begin{align}\label{Type_VI}
        \int_X f_{l_1}\overbar{f_{l_2}} \cdots f_{l_{2r-1}}\overbar{f_{l_{2r}}} d\mu = 0 \,\,\, \text{whenever $l_1,\dots, l_{2r}$ are all distinct}.
    \end{align}
    Then we have
    \begin{align*}
        \norm{\sum_{l=1}^L  f_l}_{2r} \leq C_r \norm{\left(\sum_{l=1}^L\abs{f_l}^2\right)^{\frac{1}{2}}}_{2r}.
    \end{align*}
    In particular, we may take $C_r < 4r$. The order 
    $O(r)$ is optimal.
\end{thm}
\begin{proof}
    Apply Corollary \ref{identity_multilinear} with $n=2r$, $V_j=\F=\C$, $v_{j,l} = f_{l}$ ($\forall\,j\in[2r], 1\leq l\leq L$), and $\Lambda: (v_1,\cdots,v_{2r})\mapsto v_1\overbar{v_2}\cdots v_{2r-1}\overbar{v_{2r}}$. Then we have
    \begin{align}\label{apply}
        \sum_{\substack{(l_j)_{j\in[2r]}\\l_j\text{'s all distinct}}} f_{l_1}\overbar{f_{l_2}}\cdots f_{l_{2r-1}}\overbar{f_{l_{2r}}} & = \sum_{\PP\in\FF}\DD(\PP)\sum_{(l_P)_{P\in\PP}} f_{l_{\PP(1)}}\overbar{f_{l_{\PP(2)}}}\cdots f_{l_{\PP(2r-1)}}\overbar{f_{l_{\PP(2r)}}}\nonumber\\
        & = \abs{\sum_{l=1}^L f_l}^{2r} + \sum_{\PP>\PP_0}\DD(\PP)\sum_{(l_P)_{P\in\PP}} f_{l_{\PP(1)}}\overbar{f_{l_{\PP(2)}}}\cdots f_{l_{\PP(2r-1)}}\overbar{f_{l_{\PP(2r)}}}.
    \end{align}
    Integrate both sides of (\ref{apply}) over $X$. The vanishing condition (\ref{Type_VI}) implies
    \begin{align*}
        \int_X \sum_{\substack{(l_j)_{j\in[2r]}\\l_j\text{'s all distinct}}}  f_{l_1}\overbar{f_{l_2}}\cdots f_{2r-1}\overbar{f_{2r}} d\mu = 0.
    \end{align*}
    Thus by the triangle inequality, we have
    \begin{align}\label{absolute_value}
        \int_X \abs{\sum_{l=1}^L f_l}^{2r} d\mu \leq \sum_{\PP>\PP_0}\abs{\DD(\PP)} \int_X\abs{\sum_{(l_P)_{P\in\PP}} f_{l_{\PP(1)}}\overbar{f_{l_{\PP(2)}}}\cdots f_{l_{\PP(2r-1)}}\overbar{f_{l_{\PP(2r)}}}} d\mu.
    \end{align}
    
    Then we have
    \begin{align*}
        \abs{\sum_{(l_P)_{P\in\PP}} f_{l_{\PP(1)}}\overbar{f_{l_{\PP(2)}}}\cdots f_{l_{\PP(2r-1)}}\overbar{f_{l_{\PP(2r)}}}}
        &\leq \prod_{\substack{P\in\PP\\\abs{P}= 1}} \abs{\sum_{l=1}^L f_l} \cdot \prod_{\substack{P\in\PP\\\abs{P}\geq 2}}\left( \sum_{l=1}^L\abs{f_l}^{\abs{P}} \right)\nonumber\\
        &\leq \prod_{\substack{P\in\PP\\\abs{P}= 1}} \abs{\sum_{l=1}^L f_l} \cdot \prod_{\substack{P\in\PP\\\abs{P}\geq 2}}\left( \sum_{l=1}^L\abs{f_l}^2 \right)^{\frac{\abs{P}}{2}}\nonumber\\
        &=
        \abs{\sum_{l=1}^L f_l}^{\#_1^\PP}\left(\sum_{l=1}^L\abs{f_l}^2\right)^{\frac{2r-\#_1^\PP}{2}}.
    \end{align*}
    So we can bound each integral on the right hand side of (\ref{absolute_value}) by Hölder's inequality to obtain
    \begin{align}\label{each_term}
        \int_X \abs{\sum_{l=1}^L f_l}^{2r} d\mu
        &\leq \sum_{\PP>\PP_0}\abs{\DD(\PP)}\int_X \abs{\sum_{l=1}^L f_l}^{\#_1^\PP}\left(\sum_{l=1}^L\abs{f_l}^2\right)^{\frac{2r-\#_1^\PP}{2}} d\mu\nonumber\\
        &\leq \sum_{\PP>\PP_0}\abs{\DD(\PP)} \left(\int_X \abs{\sum_{l=1}^L f_l}^{2r} d\mu\right)^{\frac{\#_1^\PP}{2r}}\left(\int_X \left(\sum_{l=1}^L\abs{f_l}^2\right)^r
        d\mu \right)^{\frac{2r-\#_1^\PP}{2r}}.
    \end{align}

    Now divide both sides of (\ref{each_term}) by $ \norm{\left(\sum_{l=1}^L\abs{f_l}^2\right)^\frac{1}{2}}_{2r}^{2r}$, and let
    \begin{align*}
        t\defeq \frac{\norm{\sum_{l=1}^L f_l}_{2r}}{\norm{\left(\sum_{l=1}^L\abs{f_l}^2\right)^\frac{1}{2}}_{2r}},
    \end{align*}
    we get
    \begin{align}\label{optimize}
        t^{2r}\leq Q(t)
    \end{align}
    for some polynomial $Q$ whose degree is $2r-2$ (as $\PP>\PP_0\Rightarrow \#_1^\PP\leq 2r-2$).

    For simplicity, from now on, we use $n$ to denote $2r$. Note that the coefficient of $t^\alpha$ ($0\leq\alpha\leq n-2$) in $Q(t)$ is
    \begin{align}\label{poly_coefficients}
        C_\alpha\defeq \sum_{ \PP>\PP_0: \#_1^\PP=\alpha}\abs{\DD(\PP)}.
    \end{align}
    
    Since $\DD(\PP)$ only depends on $((\PP))$, we can classify the terms in (\ref{poly_coefficients}) according to $((\PP))$ and apply 
    Theorem \ref{general_coefficient} to bound $C_\alpha$:
    \begin{align}
        C_\alpha &\leq \sum_{((\PP)): \#_1^\PP=\alpha} \frac{n!}{\alpha! n_1^\PP!\cdots n_m^\PP!}(n_1^\PP-1)!\cdots(n_m^\PP-1)! =\frac{n!}{\alpha!} \sum_{((\PP)): \#_1^\PP=\alpha} \frac{1}{n_1^\PP\cdots n_m^\PP}\nonumber
    \end{align}
    where we use $(n_i^\PP)_{i\in[m]}$ to represent the unordered tuple when all $1$'s are removed from $((\PP))$. Here the multinomial coefficients represents the number of ways to select a specific $\PP$ inside each $((\PP))$ (if $(n_i^\PP)_{i\in[m]}$ are not all distinct, the number is actually smaller). 
    
    For all $n$ and $\alpha$ satisfying $n-\alpha < 60$, one can directly compute that
    \begin{align}\label{program_result}
        \sum_{((\PP)): \#_1^\PP=\alpha} \frac{1}{n_1^\PP\cdots n_m^\PP} < 1.
    \end{align}
    Thus 
    \begin{align}\label{coeff_simplified_0}
        C_\alpha^{\frac{1}{n-\alpha}} < \left(\frac{n!}{\alpha!}\right)^{\frac{1}{n-\alpha}} \leq n,
    \end{align}
    where we have used the simple fact that $n!/\alpha!\leq n^{n-\alpha}$. (One can also use Stirling's formula to compute the factorials explicitly, but it turns out that such extra efforts helps little if we aim to get a unified bound on $C_\alpha^{\frac{1}{n-\alpha}}$.)

    For all $n$ and $\alpha$ satisfying $n-\alpha \geq 60$, by Proposition \ref{part_upper}, the total number of $((\PP))$ with $\#_1^\PP=\alpha$ can be bounded by $p(n-\alpha) < e^{c\sqrt{n-\alpha}}/(n-\alpha)^{3/4}$, so we have
    \begin{align}\label{coeff_bound}
        C_\alpha\leq \frac{e^{c\sqrt{n-\alpha}}}{(n-\alpha)^{\frac{7}{4}}}\frac{n!}{\alpha!}
    \end{align}
    in view of the simple fact that $n_1^\PP\cdots n_m^\PP \geq n-\alpha$ for all $((\PP))$ with $\#_1^\PP = \alpha$.
    
    Thus \begin{align}\label{coeff_simplified}
        C_\alpha^{\frac{1}{n-\alpha}} \leq \left( \frac{e^{c\sqrt{n-\alpha}}}{(n-\alpha)^{\frac{7}{4}}}\frac{n!}{\alpha!} \right)^{\frac{1}{n-\alpha}}\leq \frac{e^{\frac{c}{\sqrt{n-\alpha}}}}{(n-\alpha)^{\frac{1}{n-\alpha}\frac{7}{4}}} (n^{n-\alpha})^{\frac{1}{n-\alpha}}\leq \frac{e^{\frac{\pi}{3\sqrt{10}}}}{60^{\frac{7}{240}}}n,
    \end{align}
    where we used the fact that $e^{\frac{c}{\sqrt{n-\alpha}}}/(n-\alpha)^{\frac{1}{n-\alpha}\frac{7}{4}}$ increases monotonically for 
    $0 \leq \alpha \leq n-60$ and so achieves the maximum when $\alpha = n-60$. 
    
    Let $K\defeq e^{\frac{\pi}{3\sqrt{10}}} / 60^{\frac{7}{240}}$. Plugging (\ref{coeff_simplified_0}) and (\ref{coeff_simplified}) into (\ref{optimize}) gives
    \begin{align*}
        t^n\leq \sum_{\alpha=0}^{n-2} C_\alpha t^\alpha \leq \sum_{\alpha=0}^{n-2} (Kn)^{n-\alpha}t^\alpha.
    \end{align*}
    Divide both sides with $(Kn)^n$ and let $s=\frac{t}{Kn}$, this becomes
    \begin{align*}
        s^n\leq \sum_{\alpha=0}^{n-2}s^\alpha.
    \end{align*}
    Suppose $s^2>s+1$, then we will use induction to prove $s^n > \sum_{\alpha=0}^{n-2}s^\alpha$. The base cases can be easily verified (now we consider all $n\geq 2$). If it is true for all $n\leq n_0$, then for $n=n_0+1$ we have $s^{n_0+1} = s^{n_0-1}s^2 > s^{n_0-1}(s+1) = s^{n_0} + s^{n_0-1} > (s^{n_0-1} + s^{n_0-2}) + (\sum_{\alpha=0}^{n_0-3}s^\alpha) = \sum_{\alpha=0}^{n_0-1} s^\alpha$. So $s^n > \sum_{\alpha=0}^{n-2}s^\alpha$, a contradiction, which means we must have $s^2\leq s+1$, i.e., $s\leq \frac{1+\sqrt{5}}{2}$. Therefore, from $t=Ksn$ we get 
    \begin{align*}
        t \leq \frac{e^{\frac{\pi}{3\sqrt{10}}}}{60^{\frac{7}{240}}} \cdot \frac{1+\sqrt{5}}{2} \cdot 2r \approx 3.9992 r < 4r,
    \end{align*}
    which is what we want.

    The optimality of the order $O(r)$ can be deduced from the Burkholder-Gundy inequality (\ref{recover}) in martingale theory, which will be discussed in detail in the next subsection.
\end{proof}

If one simply uses (\ref{coeff_simplified}) to handle all $n$ and $\alpha$, then the order $O(r)$ is still available. Here we perform direct computation for $n-\alpha<60$ just to lower the absolute constant factor to $4$.

Note that the previous upper bound for $C_r$ is $2^{1/2}\left((2r)!-1\right)^{1/2}$ in \cite{gressman2023new}. It's also worth mentioning that the currently best
order of the formal constant under Type II superorthogonality is also $O(r)$, as can be seen from the proof in \cite[Section 3.1]{PierceOnSuperorthogonality}.

One can also try to apply similar arguments to get an upper bound of the formal constant in the converse inequality (Theorem \ref{main thm}). However, in this case something weird happens and there is no way to get a bound of order $O(r)$. Roughly speaking, the bound (\ref{coeff_bound}) can become $C_\alpha\lesssim n^n$ to cover the ``$\exists\abs{P}\geq3$''(joint) case in \cite{2023arXiv231201717Z}. Also, to cover the ``$\exists\abs{P}=1$''(single) case in \cite{2023arXiv231201717Z}, we need to deal with $t^n\leq \sum_{\alpha=1}^{n-2}(Kn)^{n-\alpha} t^{n-\alpha} = \sum_{\beta=2}^{n-1}(Kn)^\beta t^\beta$, which lacks homogeneity. So the final result would also be like $C_r\lesssim n^n$ ($n=2r$), which seems very bad. Such phenomenon somehow implies that bounding the formal constant of converse inequalities is more subtle than bounding that of direct inequalities.

\subsection{Applications in martingale theory and Littlewood-Paley theory}\label{subsec:applications}\phantom{x}

By a standard limiting argument as in \cite[Section 3]{gressman2023new}, Theorem \ref{effective direct} can be passed to the case when $L=\infty$. Then by vector-valued interpolation, one can easily see that the following corollary holds:
\begin{cor}\label{infinite_terms}
    Fix a $\sigma$-finite measure space $(X,d\mu)$. Assume that $\{f_l\}_{l\in\N^+}$ is a family of complex-valued functions in $L^{p}(X,d\mu)$ with $p\geq 2$, which satisfies Type IV superorthogonality for any $r\in\N^+$:
    \begin{align}\label{Type_VI}
        \int_X f_{l_1}\overbar{f_{l_2}} \cdots f_{l_{2r-1}}\overbar{f_{l_{2r}}} d\mu = 0 \,\,\, \text{whenever $l_1,\dots, l_{2r}$ are all distinct}.
    \end{align}
    Then we have
    \begin{align*}
        \norm{\sum_{l=1}^\infty  f_l}_{p} \leq C_p \norm{\left(\sum_{l=1}^\infty\abs{f_l}^2\right)^{\frac{1}{2}}}_{p}.
    \end{align*}
    In particular, if $2r$ is the smallest even number larger than $p$, then we may take $C_p < 4r \leq 2(p+2)$, i.e., $C_p = O(p)$. The order $O(p)$ is optimal.
\end{cor}

Now we consider the square function in the martingale setting. Fix a probability space $(\Omega, \mathscr{A}, \mu)$ and a sequence of $\sigma$-algebras $\mathscr{A}_1 \subseteq \mathscr{A}_2 \subseteq \cdots \subseteq \mathscr{A}_n \subseteq \cdots \subseteq \mathscr{A}$ such that $\cup_{n=1}^\infty \mathscr{A}_n = \mathscr{A}$. For a random variable $f\in L^1(\Omega, \mathscr{A}, \mu)$, we will set $E_n(f) = E(f|\mathscr{A}_n)$, $dE_n(f) = E_n(f) - E_{n-1}(f)$ for $n\in\N^+$, and $E_0 = 0$ as convention. The corresponding martingale square function is
\begin{align*}
    S(f) \defeq \left( 
    \sum_{n=1}^\infty \abs{dE_n(f)}^2 \right)^{\frac{1}{2}}.
\end{align*}

As pointed out in \cite{2023arXiv231201717Z}, martingale differences $\{dE_n(f)\}_{n\in\N^+}$ satisfy Type III (and so Type IV) superorthogonality for any $r\in\N^+$, so by Corollary \ref{infinite_terms}, we have
\begin{align*}
    \norm{f}_{p} \leq 2(p+2) \norm{S(f)}_p
\end{align*}
for $p\geq 2$.

Recall the Burkholder-Gundy inequality:
\begin{thm}[{\cite[Theorem 3.1]{burkholder1988sharp}}]\label{Burkholder}
    Let $1<p<\infty$ and $p^* = \max\{p,\frac{p}{p-1}\}$. Then
    \begin{align}\label{burkholder-gundy}
        (p^*-1)^{-1}\norm{S(f)}_p \leq \norm{f}_p \leq (p^*-1)\norm{S(f)}_p.
    \end{align}
    In particular,
    \begin{align}
        \norm{f}_p \geq (p-1)\norm{S(f)}_p & \text{ if } 1<p\leq 2,\nonumber\\
        \norm{f}_p \leq (p-1)\norm{S(f)}_p & \text{ if } 2\leq p <\infty.\label{recover}
    \end{align}
    Moreover, the constant $p-1$ is best possible.
\end{thm}
In fact, as $p\rightarrow\infty$ on the left hand side of (\ref{burkholder-gundy}), the order of the best constant is $O(p^{-1/2})$; as $p\rightarrow 1$ on the right hand side of (\ref{burkholder-gundy}), the order of the best constant is $O(1)$. See the references in \cite[Remark 3.1]{burkholder1988sharp} for details.

Therefore, our results provide an alternative proof of (\ref{recover}) up to an absolute constant factor. On the other hand, the sharpness of order $O(p)$ in (\ref{recover}) implies the sharpness of Theorem \ref{effective direct} and Corollary \ref{infinite_terms}.

One can consult (5.79) in \cite{burkholder1984boundary} for an explicit construction of the martingale which makes $p-1$ in (\ref{recover}) best possible. It's worth pointing out that there is no way to construct a sharp example for (\ref{recover}) with regular martingales (e.g., the standard dyadic martingale on $[0,1]$), in view of \cite[Lemma 2.9]{bourgain1989behavior}.

Note that Type IV superorthogonality is weaker than than Type III orthogonality, so our discussions actually show that the optimal order of the constant in the direct inequality under Type III orthogonality is also $O(r)$.

Another classical topic in harmonic analysis is the Littlewood-Paley inequality. Let $\Delta$ be the partition of $\R$ into the intervals $[2^{k-1}, 2^k)$ and $(-2^k, -2^{k-1}]$, $k\in\Z$. For any $R\in\Delta$, let $S_R$ be the Fourier projection operator associated to $R$, i.e., $\widehat{S_R(f)} = \mathbf{1}_{R}\widehat{f}$. The dyadic Littlewood-Paley square function of $f\in L^p(\R)$ is defined by 
\begin{align*}
    S^\Delta(f) \defeq \left(\sum_{R\in\Delta}\abs{S_R(f)}^2\right)^{\frac{1}{2}}.
\end{align*}
It is well known that for $1<p<\infty$, there exist two positive constants $L_{c,p}^\Delta$ and $L_{t,p}^\Delta$ such that
\begin{align}\label{cont_LP}
    (L_{t,p}^\Delta)^{-1}\norm{f}_p\leq \norm{S^\Delta(f)}_p \leq L_{c,p}^\Delta\norm{f}_p.
\end{align}
In \cite[Theorem 6.1.3]{loukas2014classical}, it's shown that both constants $L_{c,p}^\Delta$ and $L_{t,p}^\Delta$ are majorized by $O({p^*}^2)$ with $p^* = \max\{p,\frac{p}{p-1}\}$. However, the optimal orders have not been completely determined in the literature.

We can also formulate the periodic counterpart of (\ref{cont_LP}), where the corresponding dyadic partition of $\Z$ is $\widetilde{\Delta} = \{R\cap\Z: R\in\Delta\}$. We define $S_R$ in a similar way, then
\begin{align*}
    S^{\widetilde{\Delta}}(f) \defeq \left(\sum_{R\in\widetilde{\Delta}}\abs{S_R(f)}^2\right)^{\frac{1}{2}}
\end{align*}
for $f\in L^p(\T)$ where $\T\defeq[0,1]$. So (\ref{cont_LP}) becomes
\begin{align}\label{disc_LP}
    (L_{t,p}^{\widetilde{\Delta}})^{-1}\norm{f}_p\leq \norm{S^{\widetilde{\Delta}}(f)}_p \leq L_{c,p}^{\widetilde{\Delta}}\norm{f}_p.
\end{align}

By a de Leeuw type transference principle, we have
\begin{thm}[{\cite[Theorem 9]{xu2022optimal}}]
    Let $1<p<\infty$. Then the best constants in (\ref{cont_LP}) and (\ref{disc_LP}) satisfy
    \begin{align*}
        L_{c,p}^{\Delta} = L_{c,p}^{\widetilde{\Delta}} \quad \text{and}\quad L_{t,p}^{\Delta} = L_{t,p}^{\widetilde{\Delta}}.
    \end{align*}
\end{thm}

The current world records for the optimal orders of the constants $L_{c,p}^\Delta$ and $L_{t,p}^\Delta$ are as follows:
\begin{thm}[{\cite[Corollary 11]{xu2022optimal}}]
    Let $1<p<\infty$, then 
    \begin{itemize}
        \item $L_{c,p}^\Delta \approx (\frac{p}{p-1})^{3/2}$ for $1<p\leq 2$ and $L_{c,p}^\Delta \approx p$ for $2\leq p<\infty$;
        \item $L_{t,p}^\Delta \approx 1$ for $1<p\leq 2$ and $p^{1/2}\lesssim L_{t,p}^\Delta \lesssim p$ for $2\leq p<\infty$.
    \end{itemize}
\end{thm}

Bourgain \cite{bourgain1985square, bourgain1989behavior} first studied the problem on the optimal orders of the
above constants, and there have been many people working on it. One can see \cite[Section 1.2]{xu2022optimal} for comprehensive historical comments. After so many years, the only missing piece of the puzzle is the optimal order of $L_{t,p}^\Delta$ for $2\leq p <\infty$. The lower bound $p^{1/2}\lesssim L_{t,p}^\Delta$ follows from the optimal order of the best constant in the Khintchine inequality for $2\leq p<\infty$. Note that the Khintchine inequality corresponds to Type I superorthogonality, see \cite[Section 2.2]{PierceOnSuperorthogonality}. The upper bound $L_{t,p}^\Delta \lesssim p$ was proved independently by Odysseas Bakas and Hao Zhang, and the proof is recorded in \cite[Section 5]{xu2022optimal}. A previous related upper bound was $L_{t,p}^{\widetilde{\Delta}} \lesssim p\log p$ given in \cite{pichorides1990note}.

Now we come back to our purpose. As pointed out in \cite{2023arXiv231201717Z}, by dividing $\{S_R(f)\}_{R\in\Delta}$ into finitely many groups, we have Type IV superorthogonality for each group, and so Corollary \ref{infinite_terms} can be applied to give an alternative proof of $L_{t,p}^\Delta \lesssim p$. It's striking because the original proof by Odysseas Bakas and Hao Zhang involves somewhat deep harmonic analysis, while our proof relies purely on the structure of Type IV superorthogonality, which ultimately comes down to combinatorics and number theory. In other words, our proof indicates that Type 
IV superorthogonality is \textit{all} we need for the currently best order of the
constant in the reverse Littlewood-Paley inequality when $p\geq 2$. Besides, the sharpness of our main results implies that if we hope to break the $O(p)$ barrier, then we have to exploit some information in Littlewood-Paley theory \textit{beyond} Type IV superorthogonality. 

Finally, we remark that our results are quite robust as they can be equally applied to reverse Littlewood-Paley inequalities in higher dimensional cases with smooth frequency truncation functions supported on annuli. In this case, the order of the constant is still $O(p)$, which is dimension free. In contrast, if one argues as in \cite[Section 5]{xu2022optimal}, then there will be dependence on dimensions due to the use of vector-valued singular integral theory.

\subsection{Variants of Type IV superorthogonality}\label{subsec:variants}\phantom{x}

As another application of Theorem \ref{main result} (or Corollary \ref{identity_multilinear}), we introduce a new hierarchy of superorthogonality, which can be regarded as variants of Type IV superorthogonality, and prove the related converse inequality. We do not focus on formal constants in this section.
\begin{defn}[$s$-Type IV]
    Let $\{f_l\}_{1\leq l\leq L}$ be a finite family of complex-valued functions on $L^{2r}(X,d\mu)$, where $r\in\N^+$. For $s\in\N^+$ with $0\leq s \leq r$, we say $\{f_l\}_{1\leq l\leq L}$ satisfies ``$s$-Type IV'' superorthogonality for $2r$-tuples, if the vanishing property (\ref{condition}) holds whenever  the indices $l_i(1\leq i\leq 2r-2s)$, $l_{2j-1}(r-s < j\leq r)$ are all distinct, and $l_{2j-1} = l_{2j}$ for all $r-s < j\leq r$.
\end{defn}

When $s=0$, this is exactly Type IV superorthogonality. 

When $s=1$, this means that we allow one equal pair of indices in the original Type VI assumption:
\begin{align*}
    \int_X f_{l_1}\overbar{f_{l_2}} \cdots f_{l_{2r-3}}\overbar{f_{l_{2r-2}}}\abs{f_{l_{2r-1}}}^2 d\mu = 0.
\end{align*}
whenever $l_1,\cdots,l_{2r-1}$ are all distinct.

When $s=r$, this means that we have $r$ equal pair of indices, but require the pairs to be all distinct:
\begin{align*}
    \int_X \abs{f_{l_1}}^2 \abs{f_{l_3}}^2\cdots \abs{f_{l_{2r-1}}}^2 d\mu = 0
\end{align*}
whenever $l_{2j-1}$ ($j\in[r]$) are all distinct. Note that this essentially implies that there is no point belonging to all the supports of $\{f_{l_{2j-1}}\}_{j\in[r]}$.

Some simple facts are that $s$-Type IV superorthogonality is weaker than Type I$^*$/Type I/Type II superorthogonality if $s\neq r$, and $r$-Type IV superorthogonality is stronger than any other $s$-Type IV ($s\neq r$) superorthogonality.

One may have his or her own expectation regarding whether $s$-Type VI superorthogonality will become stronger or weaker as $s$ grows. However, the test sets of index tuples for different $s$-Type IV are actually disjoint, so there is no obvious relationship between different $s$-Type VI superorthogonality. Now we provide examples to show that we cannot deduce one from another in general.
\begin{exmp}\label{invent_example}
    We will construct a family of $\{f_l\}_{1\leq l\leq L}$ which is superorthogonal for $2r$-tuples of $s_0$-Type IV ($0\leq s_0 \leq r-1$) but not of any other $s$-Type IV ($0\leq s \leq r-1$, $s\neq s_0$). Note that we exclude the $s=r$ case (as we already know it's the strongest), and so $r-s\in[r]$.
    
    Let $X$ be $\T^N\defeq[0,1]^N$ ($N\in\N^+, N\geq r$) and $d\mu$ be the standard Lebesgue measure on $X$. For $k\in[N]$, let $x_k$ be the $k$-th coordinate in $\T^N$, and $v_k\in\Z^N$ be the unit vector with $k$-th coordinate 1 and other coordinates $0$. 
    
    If $r-s_0 \neq 1$, let $\B\defeq \{e^{2\pi i 2^jx_1}\}_{j\in[r-s_0-1]}$. Otherwise, let $\B=\varnothing$. 
    
    For $r-s_0< k \leq N$, construct a set of functions on $\T^N$: $\B_{k} \defeq \{e^{2\pi i 2^j2^kx_k}\}_{j\in [k-1]}\cup\{e^{2\pi i (2^k+1-2^{-k})2^kx_k}\}\cup\{e^{2\pi i (2^j+2^{-j})2^kx_k}\}_{j\in [k]}$. Let $\{f_l\}_{1\leq l\leq L}$ be $\cup_{r-s_0<k\leq N} \B_k$ plus two copies of $\B$. We claim that such $\{f_l\}_{1\leq l\leq L}$ suffices.

    Now we prove the claim. Note that our $f_l$'s are of the form $e^{2\pi i C_l\cdot x}$ ($x\in \T^N, C_l\in\Z^N$), which ensures $\abs{f_l}\equiv1$. By the Plancherel theorem, we only need to check that the set of frequencies $\{C_l\}_l$ contains an ``$(r-s,r-s)$-additive structure'' for each $s\neq s_0$, but avoids any ``$(r-s_0,r-s_0)$-additive structure''. Here an ``$(t_1,t_2)$-additive structure'' ($t_1,t_2\in[r]$) is a set of all distinct frequencies $y_1,\dots,y_{t_1}$, $z_1,\dots,z_{t_2}$ $\in\{C_l\}_l$ with $y_1+\cdots+y_{t_1} = z_1+\cdots+z_{t_2}$.

    The set of frequencies $\{C_l\}_l$ consists of $\mathcal{A}_k \defeq \{2^j2^kv_k\}_{j\in[k-1]}\cup \{(2^k+1-2^{-k})2^kv_k\} \cup\{(2^j+2^{-j})2^kv_k\}_{j\in[k]}$, $r-s_0<k\leq N$, plus two copies of $\mathcal{A} \defeq \{2^jv_1\}_{j\in[r-s_0-1]}$ if $r-s_0\neq1$. Note that two copies of $\mathcal{A}$ guarantee the existence of an $(t,t)$-additive structure for each $t\in[r-s_0-1]$ (if $r-s_0=1$ we skip this step), and the identity 
    \begin{align}\label{k-th_coord}
        \sum_{j=1}^{k-1}2^j2^kv_k + (2^k+1-2^{-k})2^kv_k = \sum_{j=1}^{k}(2^j+2^{-j})2^kv_k
    \end{align}
    guarantees the existence of $(k,k)$-additive structure for each $r-s_0<k\leq r(\leq N)$. So it remains to show that $\{C_l\}_l$ avoids any ``$(r-s_0,r-s_0)$-additive structure''.

    A key observation is that if $\mathcal{A}_k$ ($r-s_0<k\leq N$) contains a $(t_1,t_2)$-additive structure ($t_1,t_1\in\N^+$, $t_1+t_2\leq 2k$), then $t_1=t_2=k$ and the structure is given by (\ref{k-th_coord}). Here is the reason. Ignoring all common $2^k v_k$, we need to establish an identity by distinct terms in $\{2^j\}_{j\in[k-1]}\cup\{2^k+1-2^{-k}\}\cup\{2^j+2^{-j}\}_{j\in[k]}$. First focus on the fractional part. If one side has $2^k+1-2^{-k}$, then the other side must contains all $\{2^j+2^{-j}\}_{j\in[k]}$, so (\ref{k-th_coord}) must hold. If there is no $2^k+1-2^{-k}$ on both sides, but one side involves some term in $\{2^j+2^{-j}\}_{j\in[k]}$, then contradiction arises, since we can't establish an identity by distinct terms in $\{2^{-j}\}_{j\in [k]}$ in view of the base-2 expansion. Therefore, the identity can only involve terms in $\{2^j\}_{j\in[k-1]}$. However, this can't be done, still in view of the base-2 expansion. So the only possible case is (\ref{k-th_coord}), and we must have $k$ terms on both sides.

    Now we come back to show the impossibility of an ``$(r-s_0,r-s_0)$-additive structure'' in $\{C_l\}_l$. First, by linear independence of $\{v_k\}_{k\in[N]}$, this additive structure can be projected to each coordinate. If the projection along $v_k$ ($r-s_0<k\leq N$) is nontrivial, then by the observation just made, we must have $k>r-s_0$ terms on both sides, a contradiction. So both sides only involve terms from two copies of $\mathcal{A}$. But $2\abs{\mathcal{A}}=2(r-s_0-1) <2(r-s_0)$, so the number of terms are not enough, which is still a contradiction.

    Thus the proof of our previous claim is completed.
\end{exmp}
\begin{rmk}
    Taking $N=\infty$, we can extend Example \ref{invent_example} to the compact abelian group $\T^\infty$, with $\{f_l\}_l$ interpreted as characters on it.
\end{rmk}

In particular, Example \ref{invent_example} implies that we can't deduce Type IV superorthogonality from $s$-Type IV ($0<s<r$), and vice versa. It's also worth pointing out a classical example:
\begin{exmp}[\cite{gressman2023new} Haar functions]
    Let $\LL$ denote the set of standard dyadic intervals $I=[2^k \ell, 2^k (\ell+1))$ in $\R$ with $k,\ell\in\Z$. Each $I\in\LL$ is associated with Haar function
    \begin{align*}
        \psi_I = \abs{I}^{-1/2}(\mathbf{1}_{I_l}-\mathbf{1}_{I_r}),
    \end{align*}
    where $I_l,I_r\in\LL$ denote the left and the right children of $I$, respectively. Then the family of Haar functions $\{\psi_I\}_{I\in\LL}$, indexed by the set of dyadic intervals, is superorthogonal of Type IV on $\R$. If $I\subseteq [0,1]$, $\forall\,I\in\LL$, then Type III also holds. However, $s$-Type IV fails if $s\neq0$, since we can always take $f_{2r-1}$ to be a Haar function whose associated dyadic interval is contained in all the others, which forces the integral in the vanishing condition (\ref{condition}) to be $1$($\neq0$).
\end{exmp}

Now we briefly discuss some inequalities under $s$-Type IV superorthogonality. Direct inequalities seems hard to obtain by arguments in \cite{2023arXiv231201717Z}, as each term is mixed and there is no way to extract a term $\abs{\sum_{l=1}^L f_l}^{2r}$. However, the proof the converse inequality in \cite{2023arXiv231201717Z} still works:
\begin{thm}\label{converse_variant}
    Fix $r\in\N^+$ and a $\sigma$-finite measure space $(X,d\mu)$. Assume that $\{f_l\}_{1\leq l\leq L}$ is a finite family of complex-valued functions in $L^{2r}(X,d\mu)$, which satisfies $s$-Type IV superorthogonality ($0\leq s\leq r$). Then we have
    \begin{align*}
        \norm{\left(\sum_{l=1}^L\abs{f_l}^2\right)^{\frac{1}{2}}}_{2r}
        \lesssim_r 
        \norm{\sum_{l=1}^L f_l}_{2r},
    \end{align*}
    provided that (\ref{necessary}) holds.
\end{thm}
\begin{proof}
    Apply Corollary \ref{identity_multilinear} with $n=2r$, $V_j=\F=\C$, $v_{j,l} = f_{l}$ ($\forall\,j\in[2r], 1\leq l\leq L$), $\PP_1 = \cup_{1\leq i\leq 2r-2s}\{\{i\}\}\cup\cup_{r-s<j\leq r}\{\{2j-1,2j\}\}$, and $\Lambda: (v_1,\cdots,v_{2r})\mapsto v_1\overbar{v_2}\cdots v_{2r-1}\overbar{v_{2r}}$. Then we have
    \begin{align}\label{apply'}
        \sum_{(l_P)_{P\in\PP_1}}^* f_{l_{\PP_1(1)}}\overbar{f_{l_{\PP_1(2)}}}\cdots f_{l_{\PP_1(2r-1)}}\overbar{f_{l_{\PP_1(2r)}}}
        & = \sum_{\PP\geq\PP_1}\DD(\PP_1,\PP) \sum_{(l_P)_{P\in\PP}} f_{l_{\PP(1)}}\overbar{f_{l_{\PP(2)}}}\cdots f_{l_{\PP(2r-1)}}\overbar{f_{l_{\PP(2r)}}}.
    \end{align}
    We classify all partitions $\PP\geq \PP_1$ on the right hand side into three types:
    \begin{enumerate}
        \item $\PP$ is single if there exists $P\in\PP$ such that $\abs{P}=1$;
        \item $\PP$ is double if $\abs{P}=2$ for any $P\in\PP$.
        \item $\PP$ is joint if it's not single or double.
    \end{enumerate}

    Note that $\PP_1$ is a good partition (Definition \ref{good_part}), so Corollary \ref{useful} gives $\DD(\PP_1,\PP) = (-1)^s$ for any double partition $\PP\geq \PP_1$. Hence we can further write (\ref{apply'}) as
    \begin{align}\label{classify}
        \sum_{(l_P)_{P\in\PP_1}}^* f_{l_{\PP_1(1)}}\overbar{f_{l_{\PP_1(2)}}}\cdots f_{l_{\PP_1(2r-1)}}\overbar{f_{l_{\PP_1(2r)}}}
        & = (-1)^s\sum_{\substack{\PP\geq\PP_1\\\PP \text{ is double}}}\sum_{(l_P)_{P\in\PP}} f_{l_{\PP(1)}}\overbar{f_{l_{\PP(2)}}}\cdots f_{l_{\PP(2r-1)}}\overbar{f_{l_{\PP(2r)}}}\nonumber\\
        + &
        \sum_{\substack{\PP\geq\PP_1\\\PP \text{ is single}}}\DD(\PP_1,\PP)\sum_{(l_P)_{P\in\PP}} f_{l_{\PP(1)}}\overbar{f_{l_{\PP(2)}}}\cdots f_{l_{\PP(2r-1)}}\overbar{f_{l_{\PP(2r)}}}\nonumber\\
        + & \sum_{\substack{\PP\geq\PP_1\\\PP \text{ is joint}}}\DD(\PP_1,\PP) \sum_{(l_P)_{P\in\PP}} f_{l_{\PP(1)}}\overbar{f_{l_{\PP(2)}}}\cdots f_{l_{\PP(2r-1)}}\overbar{f_{l_{\PP(2r)}}}.
    \end{align}
    
    For double partitions, apply Lemma 7 in \cite{2023arXiv231201717Z} (the ``folding-the-circle'' trick) to throw away those non-diagonal terms, we get
    \begin{align}\label{double_part}
        \sum_{\substack{\PP\geq\PP_1\\\PP \text{ is double}}}\sum_{(l_P)_{P\in\PP}} f_{l_{\PP(1)}}\overbar{f_{l_{\PP(2)}}}\cdots f_{l_{\PP(2r-1)}}\overbar{f_{l_{\PP(2r)}}} 
        \geq
        \left( \sum_{l=1}^L \abs{f_l}^2 \right)^r.
    \end{align}

    Now integrate both sides of (\ref{classify}) over $X$. The form of $\PP_1$ coincides with our $s$-Type IV assumption, so the left hand side vanishes. Therefore, by (\ref{double_part}) and the triangle inequality, we obtain
    \begin{align}\label{sorted}
        \int_X \left( \sum_{l=1}^L \abs{f_l}^2 \right)^r d\mu \leq & \sum_{\substack{\PP\geq\PP_1\\\PP \text{ is single}}}\abs{\DD(\PP_1,\PP)}\int_X\abs{\sum_{(l_P)_{P\in\PP}} f_{l_{\PP(1)}}\overbar{f_{l_{\PP(2)}}}\cdots f_{l_{\PP(2r-1)}}\overbar{f_{l_{\PP(2r)}}}} d\mu\nonumber\\
        & + \sum_{\substack{\PP\geq\PP_1\\\PP \text{ is joint}}}\abs{\DD(\PP_1,\PP)}\int_X\abs{\sum_{(l_P)_{P\in\PP}} f_{l_{\PP(1)}}\overbar{f_{l_{\PP(2)}}}\cdots f_{l_{\PP(2r-1)}}\overbar{f_{l_{\PP(2r)}}}} d\mu.
    \end{align}

    The remaining parts of the proof, i.e., the bounds for the single/joint terms, are exactly the same as those in \cite{2023arXiv231201717Z}, so we do not repeat those arguments here.
\end{proof}
\begin{rmk}
    If we take $s=0$ in Theorem \ref{converse_variant}, then we recover Theorem \ref{main thm}.
\end{rmk}

In the proof of Theorem \ref{converse_variant}, it is important that we use Corollary \ref{identity_multilinear}, which is flexible because we do not have to confine ourselves to $\PP_1=\PP_0$. And such benefit ultimately comes from the more general version of the algebraic identity: Theorem \ref{main result}. 

One may ask if we can choose other $\PP_1$ in Corollary \ref{identity_multilinear} to yield more types of ``superorthogonality'' and their related inequalities. This can certainly be done. However, such inequalities will probably not be ``square function estimates'', and so not be of much interest. For example, in the proof of Theorem \ref{converse_variant}, a key point is that $\PP_1$ is a good partition, which ensures that it can be completed to a double partition. If $\PP_1$ itself is joint, then each $\PP\geq \PP_1$ must also be joint, and so there is no way to extract a term $\left( \sum_{l=1}^L \abs{f_l}^2 \right)^r$.

Another issue is whether our additional condition (\ref{necessary}) is ``well-posed'' in our setting of $s$-Type IV superorthogonality in Theorem \ref{converse_variant}. In other words, can we explicitly find $\{f_l\}_{1\leq l\leq L}$ that can be covered by Theorem \ref{converse_variant}, but out of the classical regime of Theorem \ref{main thm}?

The answer is affirmative. In fact, Example \ref{invent_example} suffices. We have shown the $\{f_l\}_{1\leq l\leq L}$ constructed in Example \ref{invent_example} are not superorthogonal of Type IV when $s_0\neq0$. So it remains to show that they all satisfy the additional condition (\ref{necessary}):
\begin{align*}
    \norm{\left(\sum_{l=1}^L\abs{f_l}^{2r}\right)^{\frac{1}{2r}}}_{2r}
    \lesssim_r 
    \norm{\sum_{l=1}^L f_l}_{2r}.
\end{align*}

The $r=1$ case can be easily verified by the Plancherel theorem, since each $f_l$ is of the form $\{e^{2\pi i C_l\cdot x}\}$, where $C_l\in\Z^N$ all distinct. So by interpolation, we only need to verify the $r=\infty$ case:
\begin{align*}
    \sup_{1\leq l\leq L}\norm{f_l}_{\infty}
    \lesssim_r 
    \norm{\sum_{l=1}^L f_l}_{\infty}.
\end{align*}
As $\norm{f_l}_\infty\equiv1$, things are further reduced to $1\lesssim_r\norm{\sum_{l=1}^L f_l}_{\infty}$.
For this purpose, we apply the theory of Sidon sets.
\begin{defn}[Sidon set]
    An infinite subset $\Lambda\subseteq \Z$ is called a ``Sidon set'', provided that for every trigonometric polynomial $p=\sum_{\lambda\in\Lambda}a_\lambda e^{2\pi i\lambda\cdot}$, one has $\norm{\{a_\lambda\}_\lambda}_{\ell^1(\Lambda)}\leq C(\Lambda)\norm{p}_{L^\infty(\T^1)}$, with a constant $C(\Lambda)$ independent of the trigonometric polynomial.
\end{defn}
\begin{defn}[Lacunary sequences]
    Let $\rho>1$ be arbitrary but fixed, and $\lambda\defeq\{n_j\}_{j\in\N^+}$ be an infinite increasing sequences of $\N^+$ with $n_{j+1}>\rho n_j$, $\forall\,j\in\N^+$. Then we call $\Lambda$ ``lacunary''.
\end{defn}
\begin{prop}[Corollary 6.17 in \cite{muscalu2013classical}, Theorem 3.6.6 in \cite{loukas2014classical}]\label{lacunary}
    Any lacunary $\Lambda$ is Sidon, and $C(\Lambda)$ can be taken to only depend on $\rho>1$. Any finite union of lacunary sequences.
\end{prop}
Moreover, according to the proof, if $\Lambda=\cup_{k=1}^m \Lambda_k$ is a finite disjoint union, where $\Lambda_k$'s are all lacunary with a uniform $\rho>1$, then $C(\Lambda)$ can be taken to only depend on $\rho$ and $m$. 

In general, we actually have:
\begin{prop}[Corollary 6.22 in \cite{muscalu2013classical}]\label{union_property}
    The union of two Sidon sets in $\Z$ is another Sidon set.
\end{prop}

Back to Example \ref{invent_example}, within 
$\B$ or $\B_k$ ($r-s_0<k\leq N$), we can regard it as a one-dimensional problem. When $\B\neq\varnothing$, since $\{2^j\}_{j\in[r-s_0-1]}$ is lacunary with $\rho=2$, it is Sidon by Proposition \ref{lacunary}, and so $1 \leq 2(r-s_0-1) \leq C \norm{2\sum_{f_l\in\B}f_l}_{\infty}$ with $C$ independent of $r,s_0$. For $\B_k$, since $\{2^j2^k\}_{j\in[k-1]}\cup\{(2^k+1-2^{-k})2^k\}$ is lacunary with $\rho=2$, and $\{(2^j+2^{-j})2^k\}_{j\in[k]}$ is lacunary with $\rho=\frac{4}{3}$, there union $\B_k$ is Sidon by Proposition \ref{lacunary}, and $1 \leq 2k \leq C' \norm{\sum_{f_l\in{\B_k}}f_l}_{\infty}$ with $C'$ independent of $r,s_0,k$. Note that $\B$ and $\B_k$, $r-s_0<k\leq N$, are independent with each other, as each of them only involves one coordinate. Therefore, by suitably arranging the signs of $\{f_l\}_l$ (which doesn't affect assumptions of superorthogonality), we can always make $\{f_l\}_l$ satisfy $\norm{\sum_{l=1}^L f_l}_{\infty} = \norm{2\sum_{f_l\in\B}f_l}_{\infty}+ \sum_{r-s_0<k\leq N} \norm{\sum_{f_l\in{\B_k}}f_l}_{\infty} \geq \min(\frac{1}{C},\frac{1}{C'})\gtrsim 1$. So the verification of (\ref{necessary}) is completed.

A final remark is that the converse inequality is a bit sensitive, as pointed out in \cite{2023arXiv231201717Z}, since even a slightly weaker version of Type IV superorthogonality fails. But our $s$-Type IV superorthogonality are neither weaker nor stronger than the original Type IV superorthogonality in general, so it makes sense that we can still obtain the converse inequality.

\section{Further directions}\label{sec:directions}
Although we have successfully obtained the optimal order of the formal constant in the direct inequality under Type III / Type IV superorthogonality, there are still many things we don't know:
\begin{itemize}
    \item For the formal constant in the direct inequality under Type II superorthogonality, the proof in \cite[Section 3.1]{pichorides1990note} only gives $C_r\leq r$, and the optimal order is still unknown.
    \item Our understanding of the formal constant in the converse inequality is still very limited, as briefly discussed at the end of Subsection \ref{subsec:constant}. This is a problem because the classical inequalities in martingale theory and Littlewood-Paley theory in Subsection \ref{subsec:applications} all only involve constants of very low order. It's interesting to at least get a polynomial bound in $r$, and compare it with the constants in those classical inequalities. The framework of the proof of the converse inequality in \cite{2023arXiv231201717Z} should be modified in some way.
    \item Actually, our algebraic identity (\ref{algebraic}) only uses the information of the symmetric sum of distinct terms, which is logically weaker than Type IV superorthogonality. Are there other ways to exploit Type IV superorthogonality?
    \item Theorem \ref{effective direct} is sharp in the sense that the order $O(r)$ of $C_r$ is sharp, but we don't know what the exact best $C_r$ is. Recall that Theorem \ref{Burkholder} tells us $C_r\geq 2r-1$, and Theorem \ref{effective direct} tells us $C_r \leq 4r$. 
    There are still obvious losses in our argument, but for now it seems hopeless to significantly narrow down the gap between $2$ and $4$. Combinatorial estimates like (\ref{program_result}) seems to be of interest itself, and we conjecture that it actually holds true for any $n$ and $\alpha$, but we do not know how to prove it. If this is the case, then we will reach the bound $(1+\sqrt{5})r \approx 3.236 r$, which is still far from $2r-1$. Therefore, it seems that $(1+\sqrt{5})r$ is an essential barrier in our framework.
\end{itemize}

\bibliographystyle{alpha}
\bibliography{sources}

\begin{thebibliography}{GPRY24}

\bibitem[Bou85]{bourgain1985square}
J.~Bourgain.
\newblock On square functions on the trigonometric system.
\newblock {\em Bull. Soc. Math. Belg. S\'{e}r. B}, 37(1):20--26, 1985.

\bibitem[Bou89]{bourgain1989behavior}
J.~Bourgain.
\newblock On the behavior of the constant in the {L}ittlewood-{P}aley inequality.
\newblock In {\em Geometric aspects of functional analysis (1987--88)}, volume 1376 of {\em Lecture Notes in Math.}, pages 202--208. Springer, Berlin, 1989.

\bibitem[Bur84]{burkholder1984boundary}
D.~L. Burkholder.
\newblock Boundary value problems and sharp inequalities for martingale transforms.
\newblock {\em Ann. Probab.}, 12(3):647--702, 1984.

\bibitem[Bur88]{burkholder1988sharp}
Donald~L. Burkholder.
\newblock Sharp inequalities for martingales and stochastic integrals.
\newblock Number 157-158, pages 75--94. 1988.
\newblock Colloque Paul L\'{e}vy sur les Processus Stochastiques (Palaiseau, 1987).

\bibitem[Com74]{comtet1974advanced}
L.~Comtet.
\newblock {\em Advanced combinatorics}.
\newblock D. Reidel Publishing Co., Dordrecht, enlarged edition, 1974.
\newblock The art of finite and infinite expansions.

\bibitem[dAP09]{de2009simple}
W.~de~Azevedo~Pribitkin.
\newblock Simple upper bounds for partition functions.
\newblock {\em Ramanujan J.}, 18(1):113--119, 2009.

\bibitem[GPRY24]{gressman2023new}
P.~T. Gressman, L.~B. Pierce, J.~Roos, and P.-L. Yung.
\newblock A new type of superorthogonality.
\newblock {\em Proc. Amer. Math. Soc.}, 152(2):665--675, 2024.

\bibitem[Gra14]{loukas2014classical}
L.~Grafakos.
\newblock {\em Classical {F}ourier analysis}, volume 249 of {\em Graduate Texts in Mathematics}.
\newblock Springer, New York, third edition, 2014.

\bibitem[Mez19]{mezo2019combinatorics}
I.~Mezo.
\newblock {\em Combinatorics and number theory of counting sequences}.
\newblock CRC Press, 2019.

\bibitem[MS13]{muscalu2013classical}
C.~Muscalu and W.~Schlag.
\newblock {\em Classical and multilinear harmonic analysis. {V}ol. {I}}, volume 137 of {\em Cambridge Studies in Advanced Mathematics}.
\newblock Cambridge University Press, Cambridge, 2013.

\bibitem[Pic90]{pichorides1990note}
S.~K. Pichorides.
\newblock A note on the {L}ittlewood-{P}aley square function inequality.
\newblock {\em Colloq. Math.}, 60/61(2):687--691, 1990.

\bibitem[Pie21]{PierceOnSuperorthogonality}
L.~B. Pierce.
\newblock On superorthogonality.
\newblock {\em J. Geom. Anal.}, 31(7):7096--7183, 2021.

\bibitem[Xu22]{xu2022optimal}
Q.~Xu.
\newblock Optimal orders of the best constants in the {L}ittlewood-{P}aley inequalities.
\newblock {\em J. Funct. Anal.}, 283(6):Paper No. 109570, 37, 2022.

\bibitem[{Zha}23]{2023arXiv231201717Z}
Jianghao {Zhang}.
\newblock {On Type IV superorthogonality}.
\newblock {\em arXiv e-prints}, page arXiv:2312.01717, December 2023.

\end{thebibliography}

\end{document}